\documentclass[a4paper,12pt,reqno]{amsart}
\usepackage{color}

\usepackage[applemac]{inputenc}
\usepackage[T1]{fontenc}
\usepackage{graphicx}

\usepackage{booktabs}  % For publication quality tables in LaTeX

\newcounter{notecounter}

\newtheorem{Th}{Theorem}

\theoremstyle{definition}

\begin{document}

\title{From Farey fractions to the Klein quartic and beyond}

\author{Ioannis Ivrissimtzis$^1$}
\address{$^1$Department of Computer Science, Durham University, DH1 5LE, UK}
\email{ioannis.ivrissimtzis@durham.ac.uk}

\author{David Singerman$^2$}
\author{James Strudwick$^2$}
\address{$^2$Mathematical Sciences, University of Southampton, SO17 1BJ, UK}
\email{D.Singerman@soton.ac.uk, J.Strudwick@soton.ac.uk}

\date{}

\begin{abstract}
In his 1878/79 paper \cite{K1}, Klein produced his famous 14-sided polygon representing the Klein quartic, 
his Riemann surface of genus 3 which has PSL(2,7) as its automorphism group. The construction 
and method of side pairings are fairly complicated. By considering the Farey map modulo 7 we show how 
to obtain a fundamental polygon for Klein's surface using arithmetic. Now the side pairings are immediate 
and essentially the same as in Klein's paper. We also extend this idea from 7 to 11 as Klein attempted to 
do in his follow up paper \cite{K11} in 1879.
\bigskip 

\noindent {\em 2010 Mathematics Subject Classification}. 

\smallskip 

\noindent Primary: 30F10, 20H10; Secondary: 51M20.
\end{abstract}

\maketitle

%   ************************************************************************** 
%   ***************************       SECTION  1      ******************************* 
%   **************************************************************************  

\section{introduction} 
The Klein quartic was introduced in one of Felix Klein's most famous papers, published in volume 14 of Mathematische Annalen 1878/79 \cite{K1}.  
A slightly updated version appeared in Klein's Collected Works \cite{K}, while for an English translation see the book 
{\it The Eightfold  Way, the beauty of Klein's Quartic curve} edited by Silvio Levy, \cite{L}). This algebraic curve, whose equation is 
$x^3y+y^3z+z^3x=0$, gives the compact Riemann surface of genus 3 with 168 automorphisms, the maximum number by the Hurwitz 
bound. 

Let $\mathbb{H}$ denote the upper-half complex plane and let $\mathbb{H}^*=\mathbb{H}\cup \mathbb{Q}\cup{\ \{\infty\}}.$  
Klein's surface is $\mathbb{H^*}/\Gamma(7)$,  where $\Gamma(7)$ is the principal congruence subgroup mod 7 of the classical 
modular group $\Gamma$=PSL(2, $\mathbb{Z})$. Klein studies the Riemann surface of the Klein quartic by constructing his famous 
14-sided fundamental region with its side identifications. See sections 11 and 12 of \cite{K1} for the construction and between pages 
448 and 449 of \cite{K1}, page 126 of {\cite{K}}, or page 320 of {\cite{L}} for the figure itself. 

Our approach is to construct a fundamental region for Klein's surface using the Farey tessellation ${\mathcal{M}_3}$ of $\mathbb{H}^*$, a 
triangular tessellation of $\mathbb{H}^*$ which we define in \S 2, and which it was shown to be the universal triangular tessellation 
\cite{S}. In \S 3 and \S 4, we study the {\it level n Farey map} $\mathcal{M}_3/\Gamma(n)$, through the correspondence of its directed 
edges with the elements of $\Gamma / \Gamma(n)$ and the correspondence of its vertices with the cosets of $\Gamma_1(n)$ in $\Gamma$ . 
In \S 5 and \S 6, we study the level 7 Farey map $\mathcal{M}_3/\Gamma(7)$. As $\mathcal{M}_3\subset \mathbb{H}^*$,  
$\mathcal{M}_3/\Gamma(7)\subset \mathbb{H^*}/\Gamma(7)$, this Farey map is embedded in the Klein surface. In a sense, we will show 
this Farey map {\it{is}} the Klein surface. 

In \S 7 and \S 8, we review Klein's original construction, computing Farey coordinates on Klein's 14-sided fundamental region 
and discussing the differences between the two approaches.  In volume 15 of Mathematische Annalen in 1879 \cite{K11}, 
Klein extended his work to study the surface  $\mathbb{H}/\Gamma(11)$, which has PSL(2,11) of order 660 as its automorphism 
group and is somewhat more complicated. He was unable to draw a fundamental region for $\mathbb{H}/\Gamma(11)$ as he was for
 $\mathbb{H}/\Gamma(7)$,  however we are able to draw the corresponding Farey map in \S 9.

%   ************************************************************************** 
%   ***************************       SECTION  2      ******************************* 
%   **************************************************************************  

\section{The Farey map.}

The  vertices of the Farey map $\mathcal{M}_3$ are the extended rationals, i.e. $\mathbb Q\cup\{\infty\}$ and two rationals $\frac{a}{c}$ and $\frac{b}{d}$ are joined by an edge if and only if $ad-bc=\pm 1$. These edges are drawn as semicircles or vertical lines, perpendicular to the real axis, (i.e. hyperbolic lines).
 Here $\infty=\frac{1}{0}$. This map has the following properties. 
\begin{enumerate}
\item[(a)] There is a triangle with vertices  $\frac{1}{0}, \frac{1}{1}, \frac{0}{1}$, called the principal triangle.
\item[(b)] The modular group $\Gamma$= PSL(2, $\mathbb{Z})$ acts as a group of automorphisms of $\mathcal{M}_3$.
\item[(c)] The general triangle has vertices $\frac{a}{c},\frac{a+b}{c+d}, \frac{b}{d}$.
\end{enumerate}
This forms a triangular tessellation of the upper half plane. Note that the triangle in (c) is just the image of the principal triangle under the M\"obius transformation corresponding to the matrix $\begin{pmatrix} 
  a&b\\
  c&d\\
  \end{pmatrix}$.

In \cite{S} it was shown that $\mathcal{M}_3$ is the {\it universal triangular map.}  This means that if $\mathcal M$ is any triangular map on an orientable surface then $\mathcal {M}$ is the quotient of $\mathcal{M}_3$ by a subgroup $\Lambda$ of the modular group. A map is regular if its automorphism group acts transitively on its {\it  darts}, (i.e. directed edges) and $\mathcal{M}_3/\Lambda$ is regular if and only if $\Lambda$ is a normal subgroup of $\Gamma$. The subgroup $\Lambda$ here is called a {\it map subgroup}.

\begin{figure}[t]
\begin{center}
\includegraphics[width=\linewidth]{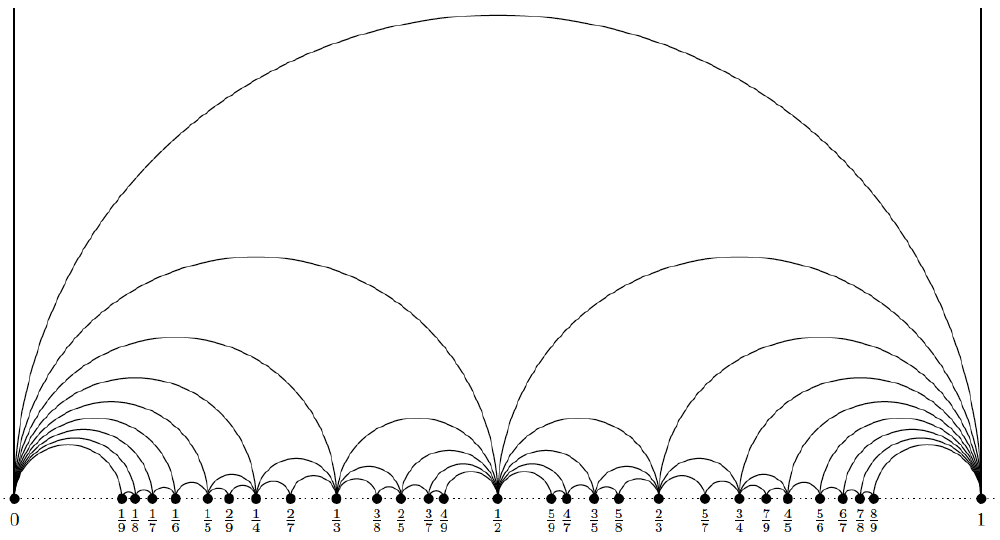}
\end{center}
\caption{The Farey map, (drawn by Jan Karaba\v s)}
\end{figure}
  (In general if  $\Delta(m,n)$ is the 
$(2,m,n)$ triangle group, then every map of type $(m,n)$ has the form $\hat{\mathcal{M}}/M$ where $\hat{\mathcal{M}}$  is the universal map of type $(m,n)$ and $M$ is a subgroup of $\Gamma$.  In  our case we are thinking of the modular group $\Gamma$ as being the $(2,3,\infty)$ group.   The infinity here means that we are not concerned with the vertex valencies;  we just require the map to be triangular.  For the general theory we refer to \cite{JS}.) 

We now consider the case when $\Lambda = \Gamma(n)$, the principal congruence mod $n$ of the modular group $\Gamma$ The corresponding maps are denoted by $\mathcal{M}_3(n)$. As $\Gamma(n)$ is a normal subgroup these maps are regular. 

%   ************************************************************************** 
%   ***************************       SECTION  3      ******************************* 
%   ************************************************************************** 

\section{The map $\mathcal{M}_3(n)$} 

The map $\mathcal{M}_3(n)$ is a regular map that lies on the Riemann surface $\mathbb{H}/\Gamma(n)$.  
The  automorphism group of $\mathcal{M}_3(n)$  is $\Gamma/\Gamma(n)\cong $PSL(2, $\mathbb{Z}_n)$ of order  

$$\mu(n)=\frac{n^3}{2} \Pi_{p|n}(1-\frac{1}{p^2}).$$

Now $\mu(n)$ is the number of darts of $\mathcal{M}_3(n)$  so the number of edges of this map is $\mu(n)/2$, 
and the number of faces is equal to $\mu(n)/3$. Note that $\frac{1}{0}$ is joined to $\frac{k}{1}$ for 
$k=0\dots n-1$ so that $\frac{1}{0}$ has valency $n$ and by regularity every vertex has valency $n$. Thus 
the number of vertices is equal to $\mu(n)/n$.  For example, $\mu(5)=60, \mu(7)=168, \mu(11)=660$, so the 
numbers of vertices of $\mathcal{M}_3(n)$, for $n=5,7,11$, are 12, 24, 60, respectively. We can now use the 
Euler-Poincar\'e formula find the well-known formula for the genus $g(n)$ of $\mathcal{M}_3 (n)$; 
$$g(n)=1+\frac{n^2}{24}(n-6)\Pi_{p|n}(1-\frac{1}{p^2}) \eqno(1).$$

\subsection{Farey coordintes for $\mathcal{M}_3(n)$} If $(a,c,n)=1$ then the projection of $\frac{a}{c}$ from 
$\mathcal{M}_3$ to $\mathcal{M}_3(n)$ is denoted by $[\frac{a}{c}$], or simply  $\frac{a}{c}$ when there 
is no room for ambiguity and is called a {\it Farey fraction} or {\it Farey coordinate} when used to denote 
a vertex of $\mathcal{M}_3(n)$. If $\frac{a}{c}$ is a Farey fraction, $a,c$ can be 
thought of as elements of $\mathbb{Z}_n$, where $(a,c,n)=1$ and $\frac{a}{c} = \frac{-a}{-c}$ 
For example, in \cite{SS} we constructed $\mathcal{M}_3(5)$ which is an icosahedron, discussed below in \S 4.

%   ************************************************************************** 
%   ***************************       SECTION  4      ******************************* 
%   ************************************************************************** 

\section{The quasi-icosahedral structure of Farey maps}  

We now show that every Farey map has a {\it quasi-icosahedral} structure. Let us give some definitions from\cite{SS}.  

\begin{enumerate}
 \item The {\it (graph-theoretic) distance}  $\delta(f_1,f_2)$ between two vertices $f_1$ and $f_2$ of a graph is the length of the shortest path joining these two vertices. 
\item A {\it Farey circuit} is a sequence of Farey fractions $f_1, f_2,\dots f_k$ where $f_i$ is joined by an edge to $f_{i+1}$ with the indices taken mod $k$.
\item A {\it pole} of a Farey map is any vertex with coordinates $\frac{a}{0}$.  
\end{enumerate}

The following theorem was proved in \cite{SS}.

\begin{Th} Let $\frac{a}{c}, \frac{b}{d}$ be distinct vertices of $\mathcal{M}_3(p)$, where $p$ is prime, and let $\Delta=ad-bc$.  Then:  
\begin{displaymath}
\delta\left(\frac{a}{c},\frac{b}{d}\right)=
\begin{cases}
1 & \mbox{if and only if $|\Delta|=1$,}\\

2 & \mbox{ if and only if $|\Delta|\not= 0,\pm 1$}\\

3 & \mbox{if and only if $\Delta=0$,}\\
\end{cases} 
\end{displaymath}
\end{Th}

Now let $\frac{b}{d}=\frac{1}{0}$ and call this the {\it north pole}  $N$ of $\mathcal{M}_3(p).$  Then by the above theorem $\delta(N,\frac{a}{c}) =1$, if and only if $c=1$, $\delta(N, \frac{a}{c})=2$ if and only if $c\not= 0, \pm 1,$, and  $\delta(N, \frac{a}{c})=3$ if and only if $c=0$. That is, the vertices of $\mathcal{M}_3(p)$ form four disjoint subsets: the north pole $N$ at $\frac{1}{0}$, a circuit of lengh $n$ whose graph-theoretic distance from $N$  is 1, another circuit at distance 2 from $N$, and other poles at distance 3 from $N$.

In \cite{SS} it was also shown that $\mathcal{M}_3(n)$  has diameter 3 for all $n \geq 5$.

\subsection{The icosahedron} $\mathcal{M}_3(5)$ is an icosahedron \cite{SS} with vertex set 
$$
\left\{ \frac{1}{0}, \frac{2}{0}, \frac{0}{1}, \frac{1}{1}, \frac{2}{1}, \frac{3}{1},\frac{4}{1},\frac{0}{2},\frac{1}{2}, \frac{2}{2},\frac{3}{2}, \frac{4}{2} \right\}
$$
see Figure 2. The north pole $N$ at $\frac{1}{0}$, there is a Farey circuit of length 5 of points whose denominator is equal to 1 and have distance 1 from $N$ and a second circuit of length 5 of points whose denominator is equal to 2 and have distance 2 from $N$.  We also have the pole $\frac{2}{0}$ at distance 3 from $N$.
\begin{figure}[ht!]
\centering
\includegraphics[width=0.75\linewidth]{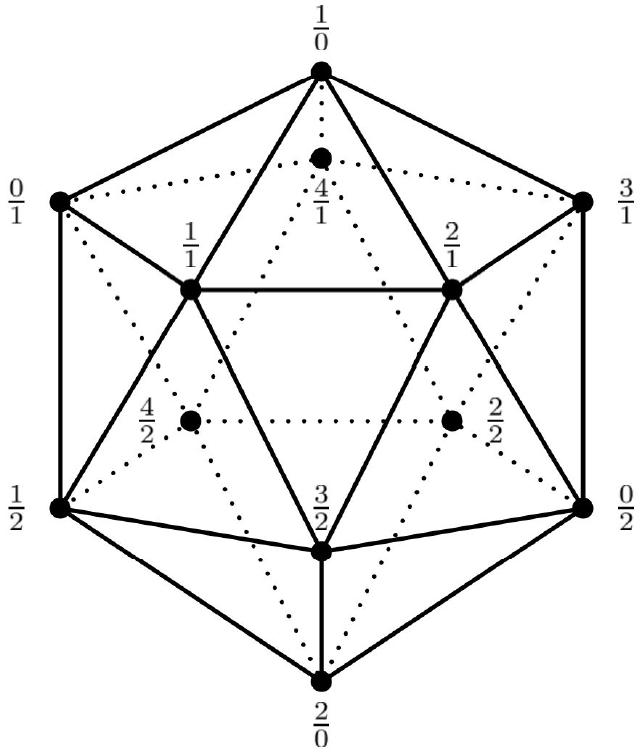}
\caption{Drawing of $\mathcal{M}_3(5)$ with Farey coordinates}
\end{figure}

\bigskip

For a {\it quasi-icosahedral} structure on $\mathcal{M}_3(p)$ let again $N=\frac{1}{0}\in \mathcal{M}_3(p)$.  
The circuit of points of distance 1 from $N$ is 
$$
S_1(p)= \frac{0}{1}, \frac{1}{1},\dots ,\frac{p-1}{1}
$$

The circuit of points at distance 2 from $N$ is more complicated and we now construct it. To make the 
calculation a bit clearer we start with the example $p=7$. From Theorem 1, we see that the points of 
distance 2 from $\frac{1}{0}$ have the form $\frac{b}{d}$ where $d=\pm 2$ or $\pm 3$. Thus the 
points $S(7)=\frac{1}{3}, \frac{1}{2}, \frac{2}{3}$ all have distance 2 from $N$.  As the transformation 
$t\mapsto t+1$ fixes $N$ and preserves distance, all points in $S(7)+k$ have distance 2 from $N$, for 
$k=1\dots ,6.$  Thus we find the Farey circuit 
$$ 
S_2(7)=
\frac{1}{3}, \frac{1}{2}, \frac{2}{3}, \frac{4}{3},\frac{3}{2},\frac{5}{3},\frac{0}{3},
\frac{5}{2},\frac{1}{3}, \frac{3}{3}, \frac{0}{2}, \frac{4}{3},\frac{6}{3},\frac{2}{2},
\frac{0}{3}, \frac{2}{3},\frac{4}{2}, \frac{3}{3},\frac{5}{3}, \frac{6}{2},\frac{6}{3}
$$
consisting points of distance 2 from $N$, see Figure 3.  We now generalize this.  Let $p\ge 5$ be a prime and let  
$$
S(p)=\frac{1}{(p-1)/2},\frac{1}{(p-3)/2}.\dots,\frac{1}{3},\frac{1}{2},\frac{2}{3},\dots, \frac{(p-3)/2}{(p-1)/2} 
$$
Then

\begin{Th}
The concatenation of sequences 
$$
S_2(p)=S (S+1) (S+2) \dots (S+p-1)
$$
where $S=S(p)$, is the Farey circuit consisting of points of distance 2 from $N$. 
The length of $S_1(p)$ is $p$ and the length of $S_2(p)$ is $p(p-4)$. 
There are $(p-1)/2$ poles.
\end{Th}

\begin{proof}
We first observe that the points in $S_2(p)$ do have distance 2 from $N$. Indeed, the points $\frac{1}{k}$ and $\frac{m-1}{m}$ 
for $2\le k,m\le \frac {p-1}{2}$  have distance 2 from $N=\frac{1}{0}$ as $\frac{1}{k} \leftrightarrow\frac{0}{1}$ and 
$\frac{m-1}{m}\leftrightarrow\frac{1}{1}$ and none of these points have distance 1 from $\frac{1}{0}$.  

The transformation $t\mapsto t+1$ fixes $\frac{1}{0}$ and preserves distance so that all points in $S+k$ have distance 2 
from $N=\frac{1}{0}$.  We now show that $S_2(p)$ is a Farey circuit. Clearly there are edges between $\frac{1}{k}$ 
and $\frac{1}{k+1}$ for $k\ge 2$ and between $\frac{k}{k+1}$ and $\frac{k+1}{k+2}$ for $k\ge 2$. So, we only need 
to show that there is an edge between the last vertex in $S+k$ and the first vertex in $S+k+1$.  The last vertex of 
$S+k$ is $k+\frac{(p-3)/2}{(p-1)/2}=\frac{(p-3+kp-k)/2}{(p-1)/2}$.  
The first vertex of $S+k+1$ is $k+1+\frac{1}{(p-1)/2}=\frac{(kp-k+p+1)/2}{(p-1)/2}$.

As   $[(p-3+kp-k)/2][(p-1)/2]-[(kp-k+p+1)/2][(p-1)/2]=-p+1$ we see that the last vertex of $S+k$ is adjacent to the 
first vertex of $S+k+1$. Thus,  $S_2(p)$ is a Farey chain consisting of points of distance 2 from $\frac{1}{0}.$  

Now $S_1(p)$ clearly has length $p$, and as the set $S_2(p)$ has length $p-4$, $S_2(p)$ has length $p(p-4)$. 
The poles are $\frac{1}{0}, \frac{2}{0}$,.... with $\frac{k}{0}=\frac{-k}{0}$, thus, the number of poles 
is $\frac{p-1}{2}$.
\end{proof}

%   ************************************************************************** 
%   ***************************       SECTION  5      ******************************* 
%   ************************************************************************** 

\section{Drawing $\mathcal{M}_3(7)$}

The map $\mathcal{M}_3(7)$ has 24 vertices with coordinates
$$
\left\{ \frac{1}{0}, \frac{2}{0}, \frac{3}{0}, \frac{0}{1}, \frac{1}{1},  \frac{2}{1},\frac{3}{1}, \frac{4}{1}, 
\frac{5}{1},\frac{6}{1}, \frac{0}{2}, \frac{1}{2}, \frac{2}{2},\frac{3}{2}, \frac{4}{2},\frac{5}{2},
\frac{6}{2},\frac{0}{3}, \frac{1}{3}, \frac{2}{3},  \frac{3}{3},\frac{4}{3},\frac{5}{3},\frac{6}{3} \right\}.
$$
The first circuit is $S_1(7)=\frac{0}{1}, \frac{1}{1}, ...,\frac{6}{1}$, and we draw a circle $C_1(7)$, centre 
$\frac{1}{0}$  and radius 1, containing it. We draw a circle $C_2(7)$, centre $\frac{1}{0}$ radius 2, containing 
the points of $S_2(7)$.  In Figure 3, $C_2(7)$ passes through the points $\frac{1}{3}, \frac{6}{3}, \frac{6}{2}, \frac{5}{3}\cdots$. 

Finally, we can draw a simple close curve $C_3(7)$ exterior to both $C_1(7)$ and $C_2(7)$ which contains the poles 
$\frac{2}{0}$ and $\frac{3}{0}$, see the dotted line in Figure 3. The pole $\frac{2}{0}$ is a vertex of seven triangles 
whose base is on the second circuit. One of these triangles is $\frac{6}{3},\frac{2}{0},\frac{1}{3}$ and the other 
6 are found by rotation, i.e., by adding 1,2,3,4,5,6 to these three points. The pole $\frac{3}{0}$ is a vertex of seven 
quadrilaterals, which are unions of two Farey triangles, and also have one edge on $C_2(7)$.  One of these is 
$\frac{1}{3},\frac{5}{2},\frac{3}{0},\frac{1}{2}$ and we get the other six by adding 1,2,3,4,5,6. 
\begin{figure}[ht!]
\centering
\includegraphics[width=\linewidth]{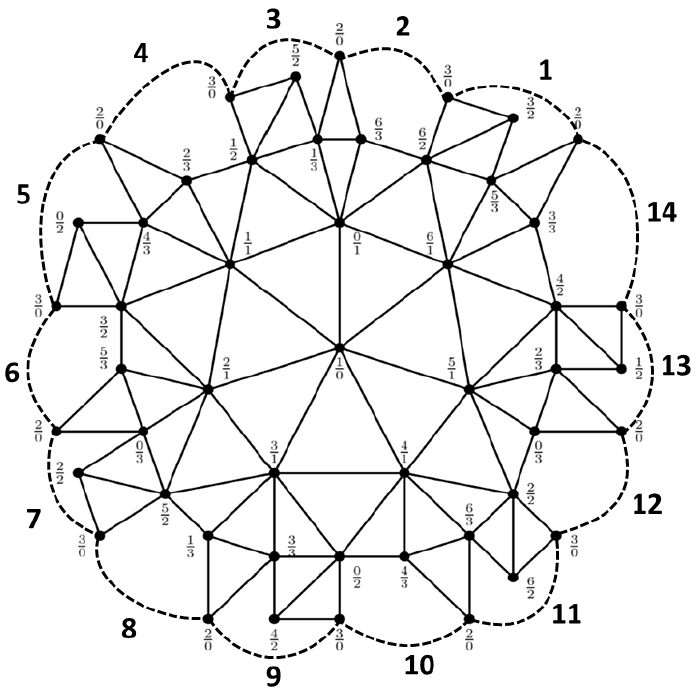}
\caption{Drawing of $\mathcal{M}_3(7)$ with Farey coordinates}
\end{figure}

%   ************************************************************************** 
%   ***************************       SECTION  6      ******************************* 
%   ************************************************************************** 

\section{The 14-sided polygon}

We now show how to obtain a 14-sided polygon out of the Farey map $\mathcal{M}_3(7)$ with 
the same side-pairings as Klein's polygon.

We  join each vertex with label $\frac{2}{0}$ to the next instance of $\frac{3}{0}$ on $C_3(7)$ 
in an anticlockwise direction, and similarly from $\frac{3}{0}$ to $\frac{2}{0}$.  We then get a 
14-sided polygon the sides of which are identified according to their Farey coordinates. Consider 
the side labelled 1. This surrounds vertices, in an anticlockwise order, with Farey coordinates 
$\frac{2}{0}$, $\frac{5}{3}$, $\frac{3}{2}$ and $\frac{3}{0}$. The side labelled 6 surrounds 
vertices in a clockwise order with exactly the same Farey coordinates. Thus we identify the side 
1 with the side 6, orientably. Similarly, we identify 3 with 8, 5 with 10, 7 with 12, 9 with 14, 
11 with 2 and 13 with 4. 

This is exactly the same side-pairing as found by Klein from his 14-sided polygon which shows that 
our 14-sided polygon does give the Klein quartic. Our way of finding the side identifications is much 
more straightforward than the method used in Klein's paper which we will summarize in \S 8.

%   ************************************************************************** 
%   ***************************       SECTION  7      ******************************* 
%   ************************************************************************** 

\section{Farey Coordinates for the Klein map}

 The Klein map $\mathcal{K}$ is the drawing of his Riemann surface of genus 3. It is the 14-sided 
polygon drawn on page 126 of \cite{K}, or page 320 of \cite{L}. See Figure 4. It has 24 vertices 
of valency 7 and we want to assign one of the 24 Farey coordinates mod 7 to each vertex. First,  
we assign the Farey coordinate $\frac{1}{0}$ to the centre point.  We note that there are two 
circuits of seven vertices centred at $\frac{1}{0}$. We give the first circuit the Farey coordinates 
$\frac{0}{1},  \frac{1}{1},\dots \frac{6}{1}$. If we extend the perpendicular bisector from 
$\frac{1}{0}$ to the hyperbolic line between  $\frac{0}{1}$ and $\frac{1}{1}$ we get to another 
vertex of valency seven to which we assign the coordinate $\frac{0+1}{1+1} = \frac{1}{2}$. 
Similarly, we extend 
the perpendiculsr bisector from $\frac{1}{0}$ to the hyperbolic line between $\frac{1}{1}$ and 
$\frac{1}{2}$ to a vertex of valency seven which we give the Farey coordinate $\frac{3}{2}$. 
By continuing, we find all vertices with Farey coordinates 
$\frac{1}{2}, \frac{3}{2}, \frac{5}{2}, \frac {0}{2}, \frac{2}{2}, \frac{6}{2}$.  Thus we have 
now found all vertices with Farey coordinates $\frac{x}{i}$ for $i=1,2$ and we just have to find 
the vertices with Farey ccordinates $\frac{x}{0}$ or $\frac{x}{3}$ which lie on the boundary of 
$\mathcal{K}$.  After Klein's identifications shown in Figure 3, we see that the 14 corners of 
$\mathcal{K}$ belong in to two classes, which we can label $\frac{2}{0}$, $\frac{3}{0}$.  
Between any two of these vertices there is precisely one more vertex, thus giving seven 
more vertices of valency 7. We  can assign Farey coordinates of the form $\frac{x}{3}$ just by 
reading them off from figure 3. In fact, each $\frac{x}{3}$ occurs exactly twice and we can 
now pair sides of $\mathcal{K}$ that have exactly the same value of $x$. Again, this gives exactly the 
same side pairing as Klein found. We thus have two methods, in sections 7 and 8, of using Farey 
coordinates to get Klein's pairings just by observation. 

Figure 4 gives a description of Klein's work using Farey coordinates. We see that each of the 14 sides 
of the boundary of $\mathcal{K}$ consists of a Farey edge and a non-Farey edge. The segment from 
$\frac{2}{0}$ to $\frac{x}{3}$ is a Farey edge whilst the segment from $\frac{x}{3}$ to $\frac{3}{0}$ 
is not a Farey edge. There is no automorphism of $\mathcal{K}$ mapping one segment to the other 
since all elements of $\Gamma$ map Farey edges to Farey edges.

\begin{figure}[ht!]
\centering
\includegraphics[width=\linewidth]{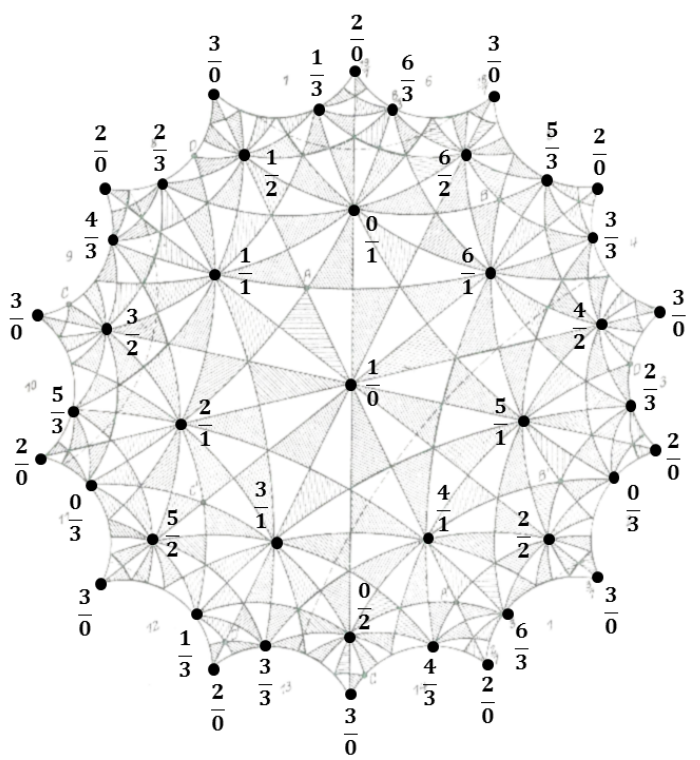}
\caption{Farey coordinates on the Klein surface}
\end{figure}

%   ************************************************************************** 
%   ***************************       SECTION  8      ******************************* 
%   ************************************************************************** 

\section{What Klein did}

Here we review Klein's original construction of his fundamental domain and show how this 
construction can be interpreted in terms of the Farey machinery we described above. 

By the end of section 10 of \cite{K} Klein had obtained the equation of his quartic curve 
and in section 11 he started to discuss the Riemann surface of this algebraic curve and 
also the corresponding map. In fact, this was one of the first publications to use maps 
(or in today's language {\it dessins d'enfants}) in a profound way, pointing up the deep 
correspondence between maps and algebraic curves. This correspondence was not 
properly understood until Grothendieck's {\it Esquisse d'un programme} some 105 years 
later \cite{G}.

Klein's quite complicated construction of the fundamental domain of his surface comes 
from considering fundamental regions for subgroups of indices 7 and 8 in the modular group.  
In section 12, he writes "In order not to make these considerations too abstract I will 
resort to the $\omega$-plane;" this is the upper-half plane where the modular group acts.  
He constructs a hyperbolic polygon corresponding to his fourteen-sided polygon describing 
his surface. He then draws semicircles (hyperbolic lines) in the upper-half plane with 
rational vertices, which correspond to the edges of his 14-sided polygon. Consider now this polygon 
as being inscribed in the unit disc so the vertices all lie on the boundary circle.  As the unit 
disc is conformally equivalent to the upper-half plane the boundary circle corresponds to 
the real axis and so, every point of the circle has some real coordinate.  He starts with one edge 
(labelled 1) of his fourteen-sided polygon corresponding to two consecutive edges of the 
polygon in the upper-half plane with vertices $\frac{2}{7}, \frac{1}{3}$ and 
$\frac{1}{3},\frac{3}{7}$. (As we already noted above, $\frac{2}{7}, \frac{1}{3}$ is a 
Farey edge while $\frac{1}{3},\frac{3}{7}$ is not, therfore we cannot map one to the 
other by an element of $\Gamma$). A second edge (labelled 6) is given as the pair 
of consecutive edges $\frac{18}{7}, \frac{8}{3}$ and $\frac{8}{3},\frac{19}{7}$. The 
M\"obius transformation corresponding to the matrix 
$$
\begin{pmatrix}113&-35\\42& -13\end{pmatrix}
$$
 in $\Gamma(7)$ maps edge 1 (i.e. $\frac{2}{7}, \frac{1}{3}, \frac{3}{7}$) to edge 6 
(i.e. $\frac{18}{7}, \frac{8}{3}, \frac{19}{7}$), and one more explicit example of edge pairing 
is given. He stated that in total seven such matrices can be found that give all the side pairings. 
We feel that our technique of just using Farey coordinates is much easier.

%   ************************************************************************** 
%   ***************************       SECTION  9      ******************************* 
%   ************************************************************************** 

\section{$\mathcal{M}_3(11)$}

About a year after Klein wrote his paper \cite{K11} on the quartic curve, he wrote a further 
paper with exactly the same title but with 7 replaced by 11, basically, he was considering 
$\mathbb{H^*}/\Gamma(11)$. In that paper he did not draw a diagram of the fundamental 
region equivalent to his drawing of  $\mathcal{K}$ in \cite{K1}. Here we show how to draw 
the Farey map  $\mathcal{M}_3(11)$ in a similar way as we drew $\mathcal{M}_3(7)$. 
This Farey map will be embedded  in the surface $\mathbb{H^*}/\Gamma(11)$. 

The first circuit of vertices of distance 1 from $\frac{1}{0}$ is 
$$
S_1(11)=\frac{0}{1}, \frac{1}{1}, \dots \frac{10}{1}.
$$
The consider the sequence of vertices 
$$
S(11)=\frac{1}{5},\frac{1}{4},\frac{1}{3},\frac{1}{2},\frac{2}{3},\frac{3}{4},\frac{4}{5}
$$
and then the second circuit is 
$$
\mathcal{S}_2(11) =S(11)\ (S(11)+1) \dots (S(11)+10).
$$

The automorphism group of  $\mathcal{M}_3(11)$ is PSL(2,11) of order 660 so the Farey map 
$\mathcal{M}_3(11)$ has 660/2=330 edges, 660/3=220 triangles and 660/11=60 vertices. The 
Farey coordinates of the vertices are $\frac{1}{0}, \frac{2}{0}, \frac{3}{0}, \frac{4}{0}, \frac{5}{0}$ 
and all Farey fractions of the form $\frac{r}{s}$ for $r$ =$0$ to $10$ and $s$=0 to 5.

To draw the map we just need to find the 220 triangular faces.  Because $z\mapsto z+1$ is an 
automorphism of  $\mathcal{M}_3(11)$, which acts as a rotation about the centre $\frac{1}{0}$ 
of the map, we see that this map is divided into eleven congruent sectors each containing 
$220/11=20$ triangles. We construct a sector, starting from the central trangle 
$\frac{1}{0}, \frac{0}{1}, \frac{1}{1}$ and adding 19 distinct triangles whose vertices lie 
$S(11)$. The 8 of these 19 triangles have a vertex on the first circuit $S_1(11)$ ($\frac{0}{1}$ or 
$\frac{1}{1}$ in particular and are uniquely determined. For the 11 triangles in the outer region 
of $C_2(11)$ there are several choices satisfying the condition that are distinct under rotation 
about $\frac{1}{0}$. 

Figure 5 shows one such solution as the union of the closures of these 20 triangles. The actual Farey coordinates are 
\begin{center}
\begin{tabular}{cccccc} 
$P_0 = \frac{1}{0}$ & $P_1 = \frac{0}{1}$ & $P_2 = \frac{1}{5}$ & $P_3 = \frac{1}{4}$ & $P_4 = \frac{1}{3}$ & $P_5 = \frac{2}{5}$ \\ \\ 
$P_6 = \frac{1}{2}$ & $P_7 = \frac{5}{0}$ & $P_8 = \frac{6}{2}$ & $P_9= \frac{7}{4}$ & $P_{10} = \frac{3}{5}$ & $P_{11} = \frac{6}{3}$ \\ \\ 
$P_{12} = \frac{4}{0}$ & $P_{13} = \frac{2}{3}$ &  $P_{14} = \frac{6}{4}$ & $P_{15} = \frac{3}{0}$ & $P_{16}= \frac{3}{4}$ & $P_{17} = \frac{4}{2}$ \\ \\ 
$P_{18} = \frac{4}{5}$ & $P_{19} = \frac{2}{0}$ & $P_{20} = \frac{6}{5}$ & $P_{21} = \frac{1}{1}$
\end{tabular}
\end{center}

\begin{figure}
\begin{center}
\includegraphics [width=0.9\linewidth]{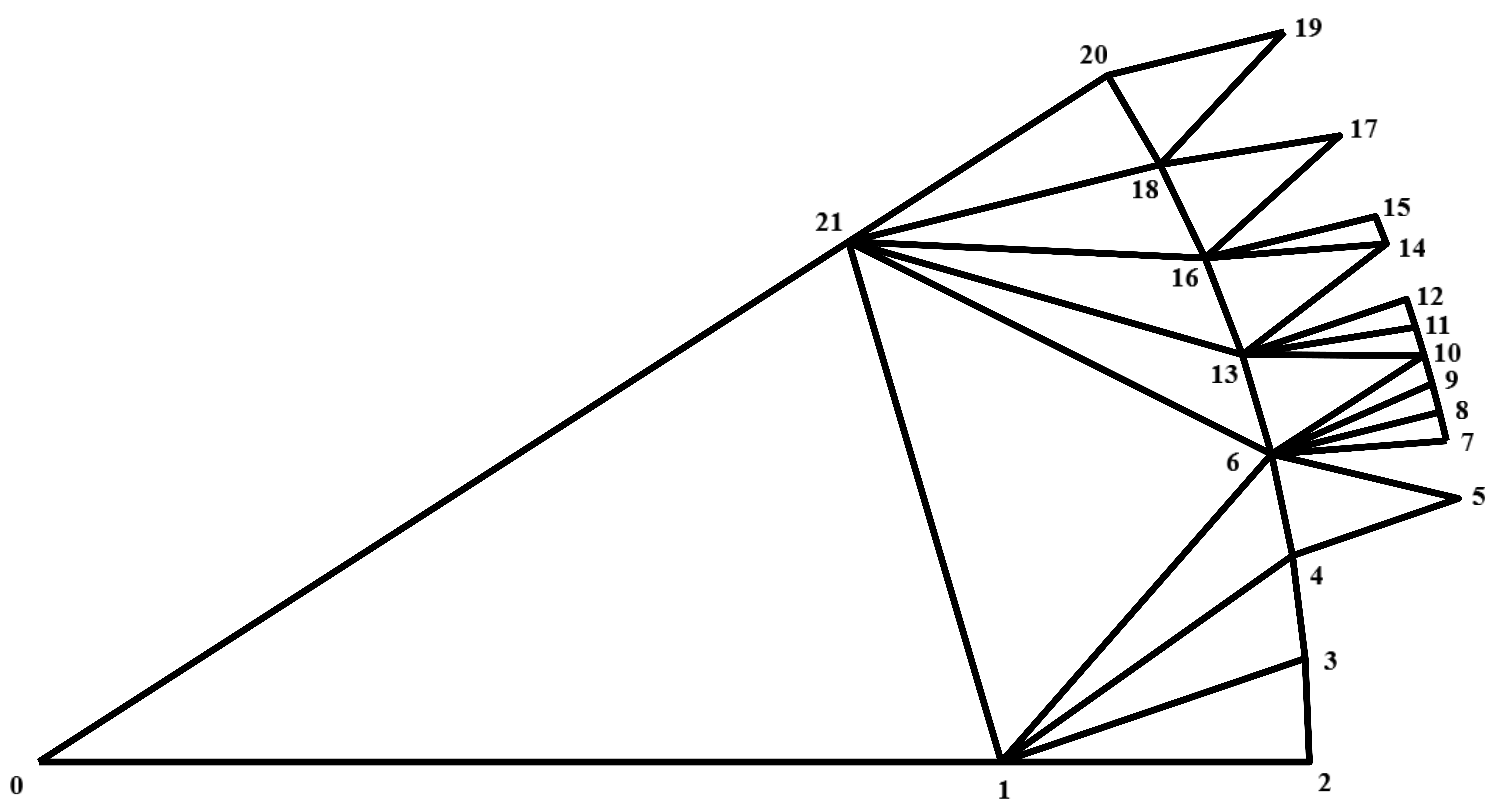}
\end{center}
\caption{The set $W$}
\end{figure}
 
Now let 
$$
W^*=W\cup (W+1)\cup \dots (W+10)
$$
where the sector $W+k$ is obtained from $W$ through the transformation $z\mapsto z+k$. 
Then $W^*$ is the union of 220 triangles as required and its boundary is a polygon with 
$11\times 18=198$. A diagram of $W^*$ is given in Figure 6. 

\begin{figure}
\begin{center}
\includegraphics[width=0.9\linewidth]{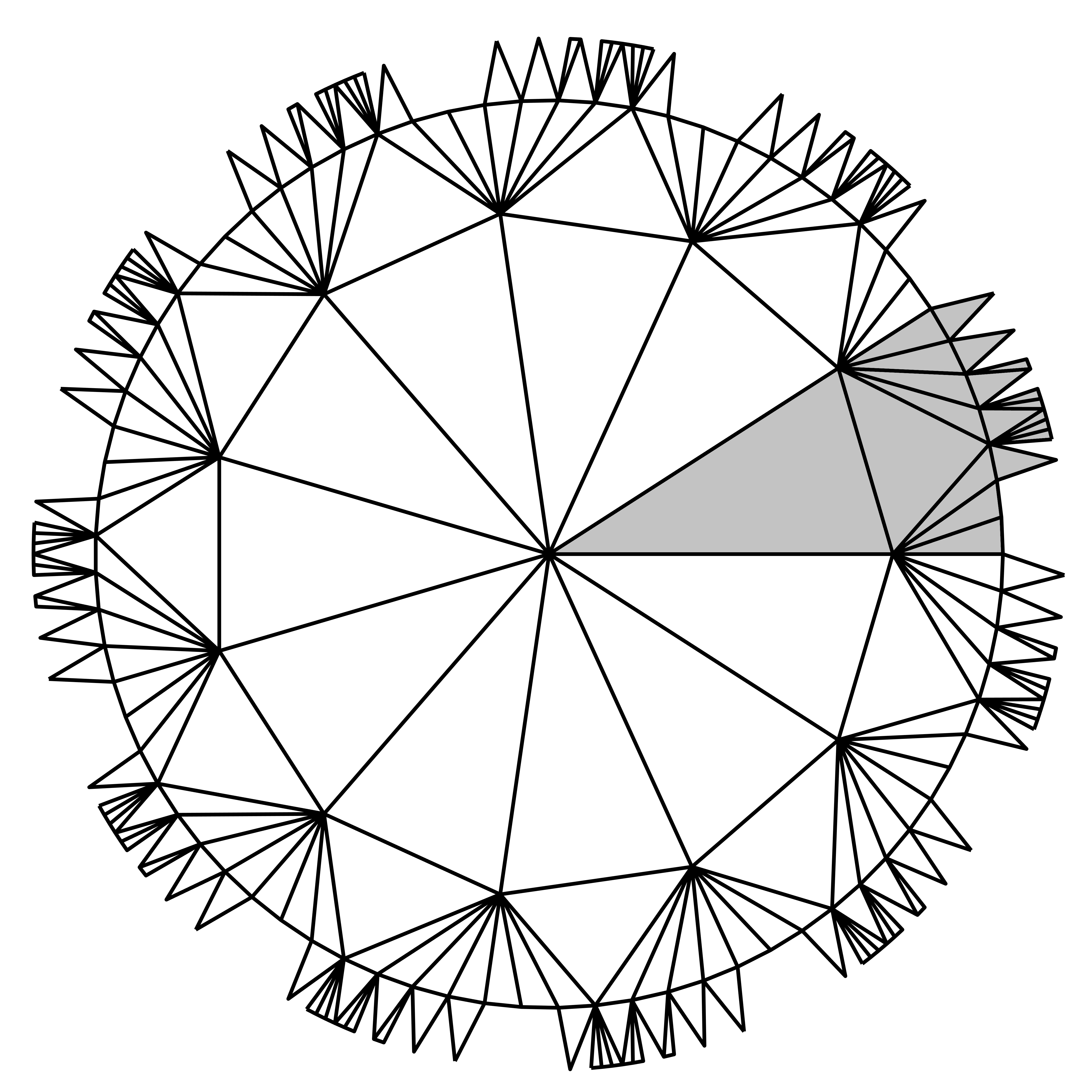}
\end{center}
\caption{The set $W^*$}
\end{figure}

Table 1 in the Appendix shows a list of the 198 boundary vertices of $W^*$ arranged in 11 columns. The $k$th 
column is just the first column  plus $(k-1)$. We now notice that we have an orientable side pairing. 
For example, the first edge 
in row 1 going from $\frac{1}{5}$ to $\frac{1}{4}$  is paired with the edge in row 5 going from 
$\frac{1}{4}$ to $\frac{1}{5}$,  the next edge in row 1 going from $\frac{1}{4}$ to $\frac{1}{3}$ 
is paired to the edge in row 8 going from $\frac{1}{3}$ to $\frac{1}{4}$.  Proceeding in this way we 
find that all the 198 edges of the polygon are paired orientably which shows that this polygon 
represents an orientable Riemann surface which must be  $\mathcal{M}_3(11)$.  As this 
corresponds to a surface subgroup of the $(2,3,11)$ triangle group  the Riemann-Hurwitz 
formula shows that this surface has genus 26.

%   ************************************************************************** 
%   **************************       REFERENCES      ******************************* 
%   ************************************************************************** 

%   ************************************************************************** 
%   ****************************       APENDIX      ******************************** 
%   **************************************************************************

\section*{Appendix}

\begin{table}[ht]
\caption{The boundary vertices of $W^*$. Each row represents a sector. The last vertex of a row is repeated as the first vertex of the row below.} \vskip 0.5cm
\setlength{\tabcolsep}{0.25em}
	\centering
		\begin{tabular}{ccccccccccccccccccc} 
		% 1 
			$\frac{1}{5}$ & $\frac{1}{4}$ & $\frac{1}{3}$ & $\frac{2}{5}$ & $\frac{1}{2}$ & $\frac{5}{0}$ & $\frac{6}{2}$ & $\frac{7}{4}$ & $\frac{3}{5}$ & $\frac{6}{3}$ & $\frac{4}{0}$ & $\frac{2}{3}$ & $\frac{6}{4}$ & $\frac{3}{0}$ & $\frac{3}{4}$ &	$\frac{4}{2}$ & $\frac{4}{5}$ & $\frac{2}{0}$ & $\frac{6}{5}$ \\ \cr 
		% 2 
			$\frac{6}{5}$ & $\frac{5}{4}$ & $\frac{4}{3}$ & $\frac{7}{5}$ & $\frac{3}{2}$ & $\frac{5}{0}$ & $\frac{8}{2}$ & $\frac{0}{4}$ & $\frac{8}{5}$ & $\frac{9}{3}$ & $\frac{4}{0}$ & $\frac{5}{3}$ & $\frac{10}{4}$ & $\frac{3}{0}$ & $\frac{7}{4}$ &	$\frac{6}{2}$ & $\frac{9}{5}$ & $\frac{2}{0}$ & $\frac{0}{5}$ \\ \cr 
		% 3 
			$\frac{0}{5}$ & $\frac{9}{4}$ & $\frac{7}{3}$ & $\frac{1}{5}$ & $\frac{5}{2}$ & $\frac{5}{0}$ & $\frac{10}{2}$ & $\frac{4}{4}$ & $\frac{2}{5}$ & $\frac{1}{3}$ & $\frac{4}{0}$ & $\frac{8}{3}$ & $\frac{3}{4}$ & $\frac{3}{0}$ & $\frac{0}{4}$ &	$\frac{8}{2}$ & $\frac{3}{5}$ & $\frac{2}{0}$ & $\frac{5}{5}$ \\ \cr 
		% 4 
			$\frac{5}{5}$ & $\frac{2}{4}$ & $\frac{10}{3}$ & $\frac{6}{5}$ & $\frac{7}{2}$ & $\frac{5}{0}$ & $\frac{1}{2}$ & $\frac{8}{4}$ & $\frac{7}{5}$ & $\frac{4}{3}$ & $\frac{4}{0}$ & $\frac{0}{3}$ & $\frac{7}{4}$ & $\frac{3}{0}$ & $\frac{4}{4}$ &	$\frac{10}{2}$ & $\frac{8}{5}$ & $\frac{2}{0}$ & $\frac{10}{5}$ \\ \cr 
		% 5 
			$\frac{10}{5}$ & $\frac{6}{4}$ & $\frac{2}{3}$ & $\frac{0}{5}$ & $\frac{9}{2}$ & $\frac{5}{0}$ & $\frac{3}{2}$ & $\frac{1}{4}$ & $\frac{1}{5}$ & $\frac{7}{3}$ & $\frac{4}{0}$ & $\frac{3}{3}$ & $\frac{0}{4}$ & $\frac{3}{0}$ & $\frac{8}{4}$ &	$\frac{1}{2}$ & $\frac{2}{5}$ & $\frac{2}{0}$ & $\frac{4}{5}$ \\ \cr 
		% 6 
			$\frac{4}{5}$ & $\frac{10}{4}$ & $\frac{5}{3}$ & $\frac{5}{5}$ & $\frac{0}{2}$ & $\frac{5}{0}$ & $\frac{5}{2}$ & $\frac{5}{4}$ & $\frac{6}{5}$ & $\frac{10}{3}$ & $\frac{4}{0}$ & $\frac{6}{3}$ & $\frac{4}{4}$ & $\frac{3}{0}$ & $\frac{1}{4}$ &	$\frac{3}{2}$ & $\frac{7}{5}$ & $\frac{2}{0}$ & $\frac{9}{5}$ \\ \cr 
		% 7 
			$\frac{9}{5}$ & $\frac{3}{4}$ & $\frac{8}{3}$ & $\frac{10}{5}$ & $\frac{2}{2}$ & $\frac{5}{0}$ & $\frac{7}{2}$ & $\frac{9}{4}$ & $\frac{0}{5}$ & $\frac{2}{3}$ & $\frac{4}{0}$ & $\frac{9}{3}$ & $\frac{8}{4}$ & $\frac{3}{0}$ & $\frac{5}{4}$ &	$\frac{5}{2}$ & $\frac{1}{5}$ & $\frac{2}{0}$ & $\frac{3}{5}$ \\ \cr 
		% 8 
			$\frac{3}{5}$ & $\frac{7}{4}$ & $\frac{0}{3}$ & $\frac{4}{5}$ & $\frac{4}{2}$ & $\frac{5}{0}$ & $\frac{9}{2}$ & $\frac{2}{4}$ & $\frac{5}{5}$ & $\frac{5}{3}$ & $\frac{4}{0}$ & $\frac{1}{3}$ & $\frac{1}{4}$ & $\frac{3}{0}$ & $\frac{9}{4}$ &	$\frac{7}{2}$ & $\frac{6}{5}$ & $\frac{2}{0}$ & $\frac{8}{5}$ \\ \cr 
		% 9 
			$\frac{8}{5}$ & $\frac{0}{4}$ & $\frac{3}{3}$ & $\frac{9}{5}$ & $\frac{6}{2}$ & $\frac{5}{0}$ & $\frac{0}{2}$ & $\frac{6}{4}$ & $\frac{10}{5}$ & $\frac{8}{3}$ & $\frac{4}{0}$ & $\frac{4}{3}$ & $\frac{5}{4}$ & $\frac{3}{0}$ & $\frac{2}{4}$ & $\frac{9}{2}$ & $\frac{0}{5}$ & $\frac{2}{0}$ & $\frac{2}{5}$ \\ \cr 
		% 10 
			$\frac{2}{5}$ & $\frac{4}{4}$ & $\frac{6}{3}$ & $\frac{3}{5}$ & $\frac{8}{2}$ & $\frac{5}{0}$ & $\frac{2}{2}$ & $\frac{10}{4}$ & $\frac{4}{5}$ & $\frac{0}{3}$ & $\frac{4}{0}$ & $\frac{7}{3}$ & $\frac{9}{4}$ & $\frac{3}{0}$ & $\frac{6}{4}$ &	$\frac{0}{2}$ & $\frac{5}{5}$ & $\frac{2}{0}$ & $\frac{7}{5}$ \\ \cr 
		% 11 
			$\frac{7}{5}$ & $\frac{8}{4}$ & $\frac{9}{3}$ & $\frac{8}{5}$ & $\frac{10}{2}$ & $\frac{5}{0}$ & $\frac{4}{2}$ & $\frac{3}{4}$ & $\frac{9}{5}$ & $\frac{3}{3}$ & $\frac{4}{0}$ & $\frac{10}{3}$ & $\frac{2}{4}$ & $\frac{3}{0}$ & $\frac{10}{4}$ &	$\frac{2}{2}$ & $\frac{10}{5}$ & $\frac{2}{0}$ & $\frac{1}{5}$ 
		\end{tabular} 
\end{table}


\begin{thebibliography}{99}

\bibitem{G}
Alexandre Grothendieck.
\newblock {Esquisse d'un programme. English translation of the 1984
  manuscript}.
\newblock {\em London Mathematical Society Lecture Note Series}, pages 5--48,
  1997.

\bibitem{JS}
Gareth~A Jones and David Singerman.
\newblock Theory of maps on orientable surfaces.
\newblock {\em Proceedings of the London Mathematical Society}, 3(2):273--307,
  1978.

\bibitem{K1}
Felix Klein.
\newblock Ueber die transformation siebenter ordnung der elliptischen
  functionen.
\newblock {\em Mathematische Annalen}, 14(3):428--471, 1878.

\bibitem{K}
Felix Klein.
\newblock {\em Gesammelte Mathematische Abhandlungen: Dritter Band: Elliptische
  Funktionen, Insbesondere Modulfunktionen Hyperelliptische und Abelsche
  Funktionen Riemannsche Funktionentheorie und Automorphe Funktionen}.
\newblock Springer Berlin Heidelberg, 1923.

\bibitem{L}
Silvio Levy.
\newblock {\em The eightfold way: the beauty of Klein's quartic curve},
  volume~35.
\newblock Cambridge university Press, 2001.

\bibitem{S}
David Singerman.
\newblock Universal tessellations.
\newblock {\em Rev. Mat. Univ. Complut. Madrid}, 1(1-3):111--123, 1988.

\bibitem{SS}
David Singerman and James Strudwick.
\newblock The Farey maps modulo $n$.
\newblock {\em arXiv preprint arXiv:1803.08851}, 2018. 
\newblock  (to appear in Acta Math. Univ. Comenianae) 

\bibitem{K11}
Felix Von~Klein.
\newblock Ueber die transformation elfter ordnung der elliptischen functionen.
\newblock {\em Mathematische Annalen}, 15(3):533--555, 1879.

\end{thebibliography}
\end{document}